\documentclass[11pt,english,oneside]{amsart}
\usepackage{color}
\usepackage{babel}
\usepackage{verbatim}
\usepackage{amsthm}
\usepackage{amstext}
\usepackage{amssymb}
\usepackage{amsbsy}
\usepackage[unicode=true,pdfusetitle,
 bookmarks=true,bookmarksnumbered=false,bookmarksopen=false,
 breaklinks=false,pdfborder={0 0 1},backref=page,colorlinks=true]
 {hyperref}

\newcommand{\menote}[1]{\marginpar{\color{blue}\tiny [MLE] #1}}

\makeatletter

\numberwithin{equation}{section}

\usepackage{ifpdf}
\usepackage[nobysame,abbrev,alphabetic]{amsrefs}

\usepackage{enumerate}
\usepackage{amsmath}
\usepackage{amscd}
\usepackage{amsfonts}
\usepackage{graphicx}
\usepackage[all]{xy}
\usepackage{verbatim}
\usepackage{gensymb}
\usepackage{mathrsfs}

\numberwithin{equation}{section}
\numberwithin{figure}{section}

\theoremstyle{plain}
\newtheorem{thm}{Theorem}[section]
\newtheorem{conjecture}[thm]{Conjecture}
\newtheorem{prop}[thm]{Proposition}
\newtheorem{cor}[thm]{Corollary}
\newtheorem{lem}[thm]{Lemma}
\theoremstyle{remark}
\newtheorem{claim}{Claim}
\newtheorem{rem}[thm]{Remark}
\newtheorem{remarks}[thm]{Remarks}

\newcommand{\Lam}{\Lambda}
\newcommand{\eps}{\epsilon}

\newcommand{\cA}{\mathcal{A}}

\newcommand{\cC}{\mathcal{C}}

\newcommand{\cH}{\mathcal{H}}

\newcommand{\cO}{\mathcal{O}}
\newcommand{\cP}{\mathcal{P}}

\newcommand{\cX}{\mathcal{X}}
\newcommand{\cY}{\mathcal{Y}}

\newcommand{\bA}{\mathbb{A}}

\newcommand{\bD}{\mathbb{D}}
\newcommand{\bG}{\mathbb{G}}
\newcommand{\bL}{\mathbb{L}}

\newcommand{\bP}{\mathbb{P}}
\newcommand{\bR}{\mathbb{R}}
\newcommand{\bZ}{\mathbb{Z}}
\newcommand{\bQ}{\mathbb{Q}}
\newcommand{\bF}{\mathbb{F}}

\newcommand{\bN}{\mathbb{N}}
\newcommand{\bH}{\mathbb{H}}

\newcommand{\bS}{\mathbb{S}}

\newcommand{\SL}{\operatorname{SL}}

\newcommand{\defi}{\overset{\on{def}}{=}}

\newcommand\norm[1]{\left\|#1\right\|}

\newcommand\set[1]{\left\{#1\right\}}
\newcommand\pa[1]{\left(#1\right)}

\newcommand\av[1]{\left|#1\right|}
\newcommand\on[1]{\operatorname{#1}}

\newcommand\mb[1]{\mathbf{#1}}

\newcommand\br[1]{\left[#1\right]}
\newcommand\smallmat[1]{\pa{\begin{smallmatrix}#1\end{smallmatrix}}}

\newcommand{\lra}{\longrightarrow}
\newcommand{\ra}{\rightarrow}

\newcommand{\onto}{\xymatrix{\ar@{>>}[r]&}}
\newcommand{\da}[4]{\xymatrix{#1 \ar@<.5ex>[r]^{#2} \ar@<-.5ex>[r]_{#3} & #4}}
\newif\ifdraft\drafttrue
\makeatother

\begin{document}

\title[Integer points and their orthogonal lattices]{Integer points on Spheres and their orthogonal lattices }

\author[]{}

\author[M. Aka]{Menny Aka}

\author[M. Einsiedler]{Manfred Einsiedler}

\author[U. Shapira]{Uri Shapira\\
(with an appendix by Ruixiang Zhang)}

\begin{abstract}
Linnik proved in the late 1950's the equidistribution of integer points on large spheres under a congruence condition. The congruence condition was
lifted in 1988 by Duke (building on a break-through by Iwaniec) using completely different techniques.
We conjecture that this equidistribution result also extends to the pairs consisting of a  vector on the sphere and the shape of the lattice
in its orthogonal complement.  We use a joining result for higher rank diagonalizable actions to obtain this conjecture
under an additional congruence condition.
\end{abstract}

\address{M.A. Departement Mathematik\\
ETH Z\"urich\\
R\"amistrasse 101\\
8092 Zurich\\
Switzerland}

\email{menashe-hai.akka@math.ethz.ch}
\thanks{M.A. acknowledges the support of ISEF, Advanced Research Grant
228304 from the ERC, and SNF Grant 200021-152819. }

\address{M.E. Departement Mathematik\\
ETH Z\"urich\\
R\"amistrasse 101\\
8092 Zurich\\
Switzerland}

\email{manfred.einsiedler@math.ethz.ch}
\thanks{M.E. acknowledges the support of the SNF Grant 200021-127145 and 200021-152819. }

\address{U.S. Department of Mathematics\\
Technion \\
Haifa \\
Israel }

\email{ushapira@tx.technion.ac.il}
\thanks{U.S. acknowledges the support of the Chaya fellowship and ISF grant 357/13.}

\address{R.Z. Departement of Mathematics\\
Princeton University\\
Fine Hall, Washington Road\\
Princeton, NJ USA 08544}
\email{ruixiang@math.princeton.edu}
\thanks{R.Z. acknowledges the support of Math Department of Princeton University}

\maketitle

\section{Introduction}

A theorem of Legendre, whose complete proof was given by Gauss in
\cite{Gauss1801}, asserts that an integer $D$ can be written as
a sum of three squares if and only if $D$ is not of the form $4^{m}(8k+7)$
for some $m,k\in\bN$. Let $\bD=\set{D\in\bN:D\not\equiv0,4,7\mod8}$
and $\bZ_{{\rm prim}}^{3}$ be the set of primitive vectors in $\bZ^{3}$.
Legendre's Theorem also implies that the set
\[
\bS^{2}(D)\defi\set{v\in\bZ_{\textrm{prim}}^{3}:\norm{v}_{2}^{2}=D}
\]
is non-empty if and only if $D\in\bD$. This important result has been refined in many ways. We are interested in the refinement
known as Linnik's problem.

Let $\bS^{2}\defi\set{x\in\bR^{3}:\norm{x}_{2}=\text{1}}$.
For a subset $S$ of rational odd primes we set
\[
\bD(S)=\set{D\in\bD:\text{for all }p\in S,\,-D{\rm \ mod\,}p\in\pa{\bF_{p}^{\times}}^{2}}.
\]
In the late 1950's Linnik \cite{Linnik68} proved that $\set{\frac{v}{{\normalcolor \norm{v}}}:v\in\bS^{2}(D)}$
equidistribute to the uniform measure on $\bS^{2}$ when $D\ra\infty$
under the restrictive assumption $D\in\bD(p)$ where $p$ is an odd prime. As we will again recall in this paper
(see equation (\ref{eq: split condition})) the condition $D\in\bD(p)$
should be thought of as a splitting condition for an associated torus
subgroup over $\bQ_{p}$, which enables one to use dynamical arguments.
Assuming GRH Linnik was able to remove the congruence condition.
A full solution of Linnik's problem 
was given by Duke \cite{Duke88} (following a breakthrough by
Iwaniec \cite{Iwaniec87}), who used entirely different methods.

In this paper we concern ourself not just with the direction of the vector~$v\in\bS^2(D)$
but also with the shape of the lattice~$\Lambda_v\defi\bZ^3\cap v^\perp$ in the orthogonal complement~$v^\perp$.
To discuss this refinement in greater detail we introduce the following notation. Fix
a copy of $\bR^{2}\defi\bR^{2}\times\set{0}$ in $\bR^{3}$.
To any
primitive vector $v\in\bS^{2}(D)$ we attach
an orthogonal lattice $[\Lambda_{v}]$ and an orthogonal grid $[\Delta_{v}]$ in~$\bR^2$
by the following procedure.

First, note that
\begin{equation}
[\bZ^{3}:\pa{\bZ v\oplus\Lambda_{v}}]=D\label{eq:covolume-1}
\end{equation}
since primitivity of $v$ implies that the homomorphism $\bZ^{3}\ra\bZ$
defined by $u\mapsto(u,v)$ is surjective and $v\oplus\Lambda_{v}$
is the preimage of $D\bZ$. Now we choose an
orthogonal transformation $k_{v}$ in ${\rm SO}_{3}(\bR)$ that maps
$v$ to $\norm{v}e_{3}$ and so maps $v^{\perp}$ to our fixed copy
of $\bR^{2}$. We rotate~$\bZ^3\cap v^\perp$ by~$k_v$ and obtain
a lattice in~$\bR^2$, which has covolume~$\sqrt{D}$ by  \eqref{eq:covolume-1}.
In order to normalize the covolume we also
multiply by the diagonal matrix $a_{v}={\rm diag}(D^{-\frac{1}{4}},D^{-\frac{1}{4}},D^{\frac{1}{2}})$.
This defines a unimodular lattice~$[\Lambda_v]$ in~$\bR^2$, which is well defined
up to planar rotations and so defines an element
\[
[\Lambda_v]\in\cX_{2}\defi{\rm SO_{2}(\bR)\setminus SL_{2}(\bR)/SL_{2}(\bZ).}
\]
We will refer to $[\Lambda_{v}]$ as ``the shape of
the orthogonal lattice'' attached to $v$.

We may still obtain a bit more geometric information from the given vector~$v$ as follows.
We choose a basis $v_{1},v_{2}$ of the
lattice $\Lambda_{v}$ such that $\det(v_{1},v_{2},v)>0$.
Choose $w\in\bZ^{3}$ with $(w,v)=1$
and let $g_{v}$ denote the matrix whose columns are $v_{1},v_{2},w$.
Note that $g_{v}\in{\rm SL}_{3}(\bZ)$ and that the set of choices
of $g_{v}$ is the coset $g_{v}{\rm ASL}_{2}(\bZ)$,
 where~${\rm ASL}_2=\set{\smallmat{g&*\\0&1}\vert g\in{\rm SL}_2}$.
Also note that the set of choices for $k_{v}$ is
the coset ${\rm Stab}_{{\rm SO}_{3}(\bR)}(e_{3})k_{v}={\rm SO}_{2}(\bR)k_{v}$.
As $a_{v}$ commutes with ${\rm SO}_2(\bR)$, we obtain the double coset
\[
[\Delta_{v}]={\rm SO}_{2}(\bR)a_{v}k_{v}g_{v}{\rm ASL}_{2}(\bZ).
\]
It does not depend on the choices made above and belongs to the space
\[
\cY_{2}\defi{\rm SO_{2}(\bR)\setminus ASL_{2}(\bR)/ASL_{2}(\bZ)},
\]
where we used that~$a_vk_vg_v\in{\rm ASL}_2(\bR)$.
Elements of the form~$[\Delta_v]$ will be refered to as ``orthogonal grids''
and can be identified with two-dimensional lattices together with a marked point on the associated torus,
defined up to a rotation.

Let $\tilde{\nu}_{D}$ denote the normalized counting measure on the
set
\[
\set{\pa{\frac{v}{\norm{v}},[\Delta_{v}]}:v\in\bS^{2}(D)}\subset\bS^{2}\times\cY_{2}.
\]
We are interested to find $A\subset\bD$ for which
\begin{equation}
\tilde{\nu}_{D}\stackrel{\text{weak}^{*}}{\lra}m_{\bS^{2}}\otimes m_{\cY_{2}}\text{ as }D\ra\infty\text{ with }D\in A\label{eq:ASL2 convergences}
\end{equation}
where $m_{\bS^{2}}\otimes m_{\cY_{2}}$ is the product of the natural
uniform measures $ $on $\bS^{2}$ and $\cY_{2}$. We propose the
following conjecture as a generalization of Linnik's problem and Theorem
\ref{Thm: Main theorem} below as a generalization of Linnik's theorem:
\begin{conjecture}
\label{Conj:ASL2 conjecture}The convergence in (\ref{eq:ASL2 convergences})
holds for the subset $A=\bD=\set{D:D\not\equiv0,4,7\mod8}$.
\end{conjecture}

Consider the natural projection $\pi:\cY_{2}\ra\cX_{2}$ induced
by the natural map $\phi:{\rm ASL_{2}}\ra{\rm SL_{2}}$.
Then $\tilde{\mu}_{D}\defi\pa{Id\times\pi}_{*}\tilde{\nu}_{D}$
is the normalized counting measure on
\[
\set{\pa{\frac{v}{\norm{v}},[\Lambda_{v}]}:v\in\bS^{2}(D)}\subset\bS^{2}\times\cX_{2}.
\]
Slightly simplifying the above problem we are interested to find
$A\subset\bD$ for which
\begin{equation}
\tilde{\mu}_{D}\stackrel{\text{weak}^{*}}{\lra}m_{\bS^{2}}\otimes m_{\cX_{2}}\text{ as }D\ra\infty\text{ with }D\in A.\label{eq:SL2 convergence}
\end{equation}

 Using two splitting conditions (see $\S$\ref{sec:Duke and joinings}) we are able to prove:
\begin{thm}[Main Theorem]\label{Thm: Main theorem}
Let $\bF$ denote the set of square free
integers and $p,q$ denote two distinct odd prime numbers. Then the
convergence (\ref{eq:SL2 convergence}) holds for $A=\bD(\set{p,q})\cap\bF$.
\end{thm}

\begin{remarks}
Our interest in the above problem arose via the work of Marklof \cite{Marklof-Frobenius} and W. Schmidt \cite{WSchmidt-sublattices} (see also \cite{EMSS}),
but as we later learned from P. Sarnak and R. Zhang, the question is closely related to
the work of Maass \cite{maass1956}.

	 Our method of proof builds on the equidistribution on~$\bS^2$ and on~$\cX_2$ (respectively on related covering spaces)
	as obtained by Linnik \cite{Linnik68} or Duke \cite{Duke88} (and in one instance more precisely
	the refinement of Duke's theorem obtained by Harcos and Michel \cite{MH2006}). 
 The crucial step is to upgrade these statements to the joint equidistribution. To achieve that 	
 we apply the recent classification
	of joinings for higher rank actions obtained by E.~Lindenstrauss and the second named author in~\cite{EL2014}.
	As such a classification is only possible in higher rank we need to require Linnik's splitting condition
	at two different primes.
		
The restriction to square-free numbers can be avoided but appears currently
	in our proof through the work of Harcos and Michel \cite{MH2006}, see also Remark \ref{Remark: not square free}.
	As Theorem~\ref{Thm: Main theorem} is assuming a splitting condition (actually two) Linnik's method \cite{Linnik68} could
	(most likely)
	be used to overcome the square-free condition. We refer also to \cite{ELMVDuke,EMV2010linnik}, where
	the Linnik method is used for slightly different problems.

 Using a break-through of
Iwaniec \cite{Iwaniec87}, it was shown by Duke \cite{Duke88} that
the congruence condition $D\in\bD(p)$ in Linnik's work is redundant.
In Conjecture \ref{Conj:ASL2 conjecture} we expressed our belief that the
congruence condition $D\in\bD(\set{p,q})$ in Theorem \ref{Thm: Main theorem}
is also superfluous. It is possible that analytic methods can again
be used to eliminate these congruence
conditions in the future although it does not seem to be a straightforward
matter. We refer to Appendix \ref{Appendix} for some findings in
this direction.	

 As we explain in $\S$\ref{sub:Proof component duke }
the equidistribution of $\set{[\Lambda_{v}]:v\in\bS^{2}(D)}$ on $\cX_{2}$
follows from a (refined) version of Duke's Theorem. In this context it
is not clear how to establish equidistribution of $\set{[\Delta_{v}]:v\in\bS^{2}(D)}$
on $\cY_{2}$ using the analytic methods. Using the methods below any
such equidistribution result on $\cY_{2}$ will imply a corresponding
convergence in (\ref{eq:ASL2 convergences}) for $A=\bD(\set{p,q})$.

 The higher dimensional analogues are more accessible. In fact working
with spheres in~$\bR^d$ we use unipotent dynamics 
in \cite{AES2014} to establish the equidistribution if~$d\geq 6$.
The cases~$d=4,5$ are slightly harder and need a mild congruence condition (namely that~$p\nmid D$ for a fixed odd prime~$p$)
for the method of~\cite{AES2014}. In an upcoming
paper \cite{ERW2015} of Ph.~Wirth, R.~R\"uhr, and the second named author the full result
is obtained for~$d=4,5$ by using effective dynamical arguments.
\end{remarks}

\textbf{Acknowledgements}: We would like to thank Elon Lindenstrauss, Philippe Michel,
and Akshay Venkatesh for many fruitful conversations over the last years on various topics
and research projects that lead to the current paper.
While working on this
project the authors visited the Israel Institute of Advanced Studies (IIAS) at the Hebrew University
and its hospitality is deeply appreciated. We thank Peter Sarnak and Ruixiang Zhang
for many conversations on these topics at the IIAS.

\section{Notation and organization of the paper}

We first fix some common notation from algebraic number theory: Let
$V_{\bQ}$ be the set of places on $\bQ$ containing all primes~$p$ and the archimedean place~$\infty$.
Let $\bZ_{p}$ denote the
$p$-adic numbers and for $S\subset V_{\bQ}$ we let $\bQ_{S}=\prod_{p\in S}^{'}\bQ_{p}$
be the restricted direct product w.r.t.\ the compact open subgroups
$\bZ_{p}$. Finally, we set $\bA_f=\prod^{'}_{p\in V_\bQ\setminus\{\infty\}}\bQ_p$,  $\widehat{\bZ}=\prod_{p\in V_\bQ\setminus\{\infty\}}\bZ_p$ and  $\bZ^{S}=\bZ\bigl[\set{\frac{1}{p}:p\in S\setminus\set{\infty}}\bigr]$.
Recall that~$\bQ=\bZ^{V_\bQ}$ is a cocompact lattice in the adeles~$\bA=\bQ_{V_\bQ}$.
The letter $e$ with or without a subscript will denote the identity
element of a group which is clear from the context.

A sequence of probability measures $\mu_{n}$ on a measurable space
$X$ is said to equidistribute to a probability measure $\mu$ as~$n\to\infty$ if
the sequence converges to $\mu$ in the $\text{weak}^{*}$ topology on
the space of probability measures on $X$. A probability measure $\mu$
is called a $\text{weak}^{*}$ limit of a sequence of measures
$\mu_{n}$ if there exists a subsequence $(n_{k})$ such that $\mu_{n_{k}}$
equidistribute to $\mu$ as $k\ra\infty$.

Given a locally compact group $L$ and a subgroup $M<L$ such that $L/M$ admits an $L$-invariant
probability measure, it is unique and we denote it by $m_{L/M}$ and
call it the \emph{uniform measure }on $L/M$.
Finally, the letter
$\pi$ (with or without some decorations) is used to denote various
projection maps whose definition will be clear from the context.
E.g.~if~$M<L$ are as above and~$K<L$ is a compact subgroup, there is a canonical
projection map~$\pi:L/M\to K\backslash L/M$ and we will still refer to~$\pi_*(m_{L/M})$
as the uniform measure on~$K\backslash L/M$.

We now give an overview of our proof of Theorem \ref{Thm: Main theorem}
and discuss the organization of the paper. In $\S$\ref{sec:orbits and duality},
we establish that the convergence (\ref{eq:SL2 convergence})
follows from an equidistribution of \char`\"{}joined\char`\"{} adelic (or $S$-adic)
torus orbits on a product of two homogeneous spaces. In $\S$\ref{sub:Proof component duke },
we use Duke's Theorem (resp.~\cite{MH2006}) to deduce that these orbits equidistribute to
a joining (see $\S$\ref{sec:Duke and joinings}
for the definition). Then, in $\S$\ref{sub:joinings} we show that this joining
must be the trivial joining. This will imply Theorem \ref{Thm: Main theorem}.

\section{Joined Adelic, $S$-adic and real torus orbits\label{sec:orbits and duality}}

In this section we show that Conjecture~\ref{Conj:ASL2 conjecture}
and Theorem \ref{Thm: Main theorem} follow from the equidistribution
of a sequence of ``adelic diagonal'' torus orbits on a product of
homogeneous spaces. We first explain this connection for Conjecture~\ref{Conj:ASL2 conjecture}, involving a homogeneous space for ${\rm ASL_{2}}$.

Let $\bG_{1}={\rm SO_{3}},\bG_{2}={\rm ASL_{2}}$ and $\bG=\bG_{1}\times \bG_{2}$, $G_{j}=\bG_{j}(\bR),\Gamma_{j}=\bG_{j}(\bZ)$
for $j=1,2$ and $G=\bG(\bR),\Gamma=\bG(\bZ)$, $K={\rm SO_{2}(\bR)}$ and fix $v\in\bS^{2}(D),D\in\bD$ throughout
this section. We wish to identify $K\setminus G_{1}\cong\bS^{2}$
so we let $k\in G_{1}$ act on $\bS^{2}$ by
the right action $(k,u)\mapsto k.u=k^{-1}u$; we
find it simpler to think of $\bS^{2}$ as row vectors and use the definition
$(k,u^{t})\mapsto k.u^t=u^{t}k$.
Note that this defines a transitive action satisfying $K={\rm Stab}_{G_1}(e_{3})$.
Recall the definition of $g_{v},k_{v},a_{v},[\Delta_{v}]$ from the
introduction and note that $e_{3}^tk_v=\norm{v}^{-1}v^t$.

Let $\textbf{S}^{2}\defi\bS^{2}/\Gamma_{1}$ and $\textbf{S}^{2}(D)\defi\bS^{2}(D)/\Gamma_{1}$
and $\mathbf{v}=v^t\Gamma_{1}$ and set $[\Delta{}_{\mathbf{v}}]=[\Delta_{v}]$
which is well-defined as $[\Delta_{\gamma. v}]=[\Delta_{v}]$ for all
$\gamma\in\Gamma_{1}$. The map $\mathbf{v}\in\textbf{S}^2(D)\mapsto\frac{\mathbf{v}}{\norm{\mathbf{v}}}\in\textbf{S}^2$
is also well-defined. It follows that the following double coset
\begin{equation}
K\times K\pa{k_{v},a_{v}k_{v}g_{v}}\Gamma_{1}\times\Gamma_{2}\label{eq:orbit of (v,lambda v)}
\end{equation}
represents the pair
\[
\pa{\frac{\mathbf{v}}{\norm{\mathbf{v}}},[\Delta_{\mathbf{v}}]}\in\textbf{S}^{2}\times\cY_{2}.
\]
Note that all the measures appearing in equation (\ref{eq:ASL2 convergences})
(resp. equation (\ref{eq:SL2 convergence})) are $\Gamma_{1}$-invariant
so if we consider their projection $\nu_{D}$ (resp. $\mu_{D}$) of
$\tilde{\nu}_{D}$ (resp. $\tilde{\mu}_{D}$) to $\textbf{S}^{2}\times\cY_{2}$
(resp. $\textbf{S}^{2}\times\cX_{2}$) we have that the convergence
(\ref{eq:ASL2 convergences}) is equivalent to
\begin{equation}
\nu_{D}\stackrel{\text{weak-}*}{\lra}m_{\textbf{S}^{2}}\otimes m_{\cY_{2}},D\ra\infty,D\in A\label{eq:ASL2 conv modulu Gamm1}
\end{equation}
and the convergence (\ref{eq:SL2 convergence}) is equivalent to
\begin{equation}
\mu_{D}\stackrel{\text{weak-}*}{\lra}m_{\textbf{S}^{2}}\otimes m_{\cX_{2}},D\ra\infty,D\in A.\label{eq:SL2 conv mod Gamma1}
\end{equation}

Roughly speaking, integral orbits on the $\bZ$-points of a variety
admitting a $\bZ$-action of an algebraic group $\bP$ may be parametrized
by an adelic quotient of the stabilizer. E.g., as we will see below,
$\Gamma_{1}$-orbits of $\bS^{2}(D)$, can be parametrized as an adelic
quotient of the stabilizer of $v$. The interested reader may consult
\cite[\S 3]{EV08}, \cite[\S 6.1]{EMV2010linnik} and \cite[Theorem 8.2]{PR94}.
The novelty here is that we consider a ``joint parametrization'' and
combine this with a recent work of the second named author with E.\ Lindenstrauss
\cite{EL2014}.

More concretely, consider the above right action of $G_{1}$ on $K\setminus G_{1}\cong\bS^{2}$
and set $\bH_{v}\defi{\rm Stab}_{\bG_{1}}(v)$. The group $\bH_{v}$
is defined over $\bZ\subset\bQ$ as $v\in\bZ^{3}$. Naturally, $k_{v}^{-1}{\rm Stab}_{G_{1}}\pa{\rm e_{3}}k_{v}=k_{v}^{-1}Kk_{v}=\bH_{v}(\bR)$.

In the proofs
below we will frequently use the ternary quadratic form $Q_0((v_1,v_2,v_3))=v_1^2+v_2^2+v_3^2=\|(v_1,v_2,v_3)\|_2^2$
for~$(v_1,v_2,v_3)$ belonging to~$\bQ^3$ or one of its completions.
The following lemma explains the congruence condition~$D\in \bD(p)$.

\begin{lem}
	Let~$v\in\bZ_p^3$  and~$D=Q_0(v)$.
We have that
\begin{equation}
-D=x^2\mbox{ for some }x\in\bZ_p\Rightarrow\bH_{v}(\bQ_{p})\text{ is a split torus}.\label{eq: split condition}
\end{equation}
\end{lem}

\begin{proof}
Let $w_1,w_2$ be a basis of the orthogonal complement of $v$ within~$\bQ_p^3$. Notice first that $\bH_{v}(\bQ_{p})\cong{\rm SO}(aX^{2}+bXY+cY^{2})$, where $a=\norm{w_1}_2^2,c=\norm{w_2}_2^2,b=2(w_1,w_2)$. The determinant of the companion matrix of~$Q_0$ w.r.t.~the basis $v,w_1,w_2$ is $1$ up-to $(\bQ_p^\times)^2$, that is, $D(ac-\frac14b^2)\in (\bQ_p^\times)^2$.
By the assumption on~$D$, $-\frac{4}{D}\in(\bQ_p^\times)^2$ so  $b^2-4ac\in (\bQ_p^\times)^2$ which shows that  $aX^{2}+bXY+cY^{2}$ is isotropic over $\bQ_p$. This implies the lemma.
\end{proof}

Similarly, consider the action of $G_{2}$ on $K\setminus G_{2}$
and note that
\[
{\rm Stab}_{G_{2}}({ K}a_{v}k_{v}g_{v})=g_{v}^{-1}k_{v}^{-1}a_{v}^{-1}Ka_{v}k_{v}g_{v}=g_{v}^{-1}\bH_{v}(\bR)g_{v}.
\]
Define the ``diagonally embedded'' algebraic torus $\bL_{v}$ by
\[
\bL_{v}(R):=\set{\pa{h,g_{v}^{-1}hg_{v}}:h\in\bH_{v}(R)}
\]
for any ring $R$. It is defined over $\bZ\subset\bQ$ as so is $\bH_{v}$
and $g_{v}\in{\rm SL_{3}}(\bZ)$.

In what follows we consider projections
of an adelic orbit onto $S$-arithmetic homogeneous spaces. In order
to define these projections note that $\bG_{1}$ and $\bG_{2}$ have
class number one, that is, for $j=1,2$ and for any $T\subset V_{\bQ}\setminus\set{\infty}$
we have
\begin{equation}
\bG_{j}\Bigl(\prod_{p\in T}\bZ_{p}\Bigr)\bG_{j}\bigl(\bZ^{T}\bigr)=\bG_{j}\bigl(\bQ_{T}\bigr).\label{eq:class number one}
\end{equation}
Indeed, for $\bG_{1}$ see \cite[\S 5.2]{EMV2010linnik} and for $\bG_{2}$
it follows from the same, well-known (see \cite{PR94}), assertions for the
simply-connected algebraic group ${\rm SL}_{2}$ and for $\bG_{a}^{2}$.
This implies that for $\set{\infty}\subset S\subset S'\subset V_{\bQ}$,
if we let $X_{j}^{S}\defi\bG_{j}(\bQ_{S})/\bG_{j}(\bZ^{S})$, $X^{S}\defi X_{1}^{S}\times X_{2}^{S}$
we have a well-defined projection map $\pi_{S',S}:X^{S'}\ra X^{S}$.
The map $\pi_{S',S}$ is given by dividing by $\bG(\prod_{p\in S'\setminus S}\bZ_{p})$
from the left and using (\ref{eq:class number one}). Now, consider
the following adelic orbit
\[
\cO_{D}^{\bA}:=(k_{v},e_{f},a_{v}k_{v}g_{v},e_{f})\bL_{v}(\bA)\bG(\bQ)\subset X^{V_{\bQ}},
\]
where~$e_f$ denotes the identity element in~$\bG_j(\widehat{\bZ})$ for~$j=1,2$.
Fix $\set{\infty}\subset S\subset V_{\bQ}$ and set $\cO_{D}^{S}\defi\pi_{V_{\bQ},S}\pa{\cO_{D}^{\bA}}$ and $\mu_{\cO_{D}^{S}}=(\pi_{V_{\bQ},S})_*(\mu_{\cO_{D}^{\bA}})$ where $\mu_{\cO_{D}^{\bA}}$ is the uniform measure on this orbit.
Although strictly speaking  $\cO_{D}^{S}$ depends on $v$ we omit $v$ from the notation as we will see below that it will not play a crucial role.

We now describe $\cO_{D}^{\infty}$. Take a complete set of representatives
$M_{v}\subset\bH_{v}(\bA_{f})$ for the double coset space
\[
\bH_{v}(\bR\times\widehat{\bZ})\setminus\bH_{v}(\bA)/\bH_{v}(\bQ)\cong \bH_{v}(\widehat{\bZ})\setminus\bH_{v}(\bA_{f})/\bH_{v}(\bQ),
\]
which is finite by \cite[Theorem 5.1]{PR94}. For $h\in M_{v}$, using
(\ref{eq:class number one}) we decompose $h=c_{1}(h)\gamma_{1}(h)^{-1}$
and $g_{v}^{-1}hg_{v}=c_{2}(h)\gamma_{2}(h)^{-1}$ with
\begin{equation}
c_{j}(h)\in\bG_{j}(\widehat{\bZ}),\gamma_{j}(h)\in\bG_{j}(\bQ),\, j=1,2.\label{eq:decomposition of h}
\end{equation}
We will use the abbreviation $\Theta_{K}\defi\set{(k,k):k\in K}$.
Moreover, let us write
 $$\cO_h\defi \Theta_K(k_{v}\gamma_{1}(h),a_{v}k_{v}g_{v}\gamma_{2}(h))\bG(\bZ)$$
 for $h\in M_v$.
\begin{prop}
\label{prop: S packet}
Let
$p:G/\Gamma\ra \pa{K\times K}\setminus G/\Gamma$
be the natural projection. Then,
\begin{enumerate}
\item\label{ustake1} $\cO_D^\infty=\bigsqcup_{h\in M_{v}} \cO_h$.
\item\label{ustake2} For any $h\in M_v$ the orbit $\cO_h$ projects under $p$ to a single point in $\on{supp}(\nu_D)$. Moreover, the correspondence
$h\mapsto p(\cO_h)$ is a bijection between $M_v$ and $\on{supp}(\nu_D)$.
\item\label{ustake3} $p_*(\mu_{\cO^D_\infty})=\nu_D$.
\end{enumerate}
\end{prop}

\begin{proof}
\eqref{ustake1} Using the set $M_{v}$ of representatives we can write $\cO_{D}^{\bA}$ as a disjoint union of
$\bL_{v}(\bR\times\widehat{\bZ})$-orbits:
$$
\cO_{D}^{\bA}=
\bigsqcup_{h\in M_{v}}(k_{v},e_{f},a_{v}k_{v}g_{v},e_{f})\bL_{v}(\bR\times\widehat{\bZ})(e_{\infty},h,e_{\infty},g_{v}^{-1}hg_{v})\bG(\bQ).
$$
Decomposing each~$h\in M_v$ and~$g_v^{-1}hg_v$ as in (\ref{eq:decomposition of h}) and using
that
\[
(\gamma_1(h),\gamma_1(h),\gamma_2(h),\gamma_2(h))\in\bG(\bQ)
\]
we arrive at
$$
\cO_{D}^{\bA}=
\bigsqcup_{h\in M_{v}}(k_{v},e_{f},a_{v}k_{v}g_{v},e_{f})\bL_{v}(\bR\times\widehat{\bZ})(\gamma_1(h),c_1(h),
\gamma_2(h),c_2(h))\bG(\bQ).
$$
Recalling that
$\pi_{V_{\bQ},\{\infty\}}$ is given by dividing by $\bG(\widehat{\bZ})$
from the left we get
\begin{equation}
\cO_{D}^{\infty}=\bigsqcup_{h\in M_{v}}(k_{v},a_{v}k_{v}g_{v})\bL_{v}(\bR)(\gamma_{1}(h),\gamma_{2}(h))\bG(\bZ).\label{eq:conjugation Theta K}
\end{equation}
As $\bG(\widehat{\bZ})\cap \bL_v(\bA_f)=\bL_v(\widehat{\bZ})$ this is indeed a disjoint union. Noting that $\Theta_{K}=(k_{v},a_{v}k_{v}g_{v})\bL_{v}(\bR)(k_{v}^{-1},(a_{v}k_{v}g_{v})^{-1})$
we arrive at \eqref{ustake1}.

\eqref{ustake2} We analyze $p(\cO_h)$ for $h\in M_v$. We first concentrate on the $\bG_1$ component. Identifying  $K\setminus G_{1}/\Gamma_{1}\cong\mathbf{S}^{2}$ we claim that $h\stackrel{\phi}{\mapsto} Kk_v\gamma_1(h)\Gamma_1$ is a well-defined bijection between $M_v$ and the set $\mathbf{S}^{2}(D)$. Indeed, it is shown in the proof of \cite[Theorem 8.2]{PR94} that under the above identification, $\phi$ is well-defined bijection between $M_{v}$ and the set of all
$\mathbf{w}\in \mathbf{S}^{2}(D)$ such that for all primes
$p$ there exists $g_{p}\in\bG_{1}(\bZ_{p})$ with $g_{p}.v=w$ for some $v\in \mathbf{v},w\in\mathbf{w}$ (where one uses the facts that  $\bG_{1}$ has class number 1 and that by Witt's Theorem $G_1(\bQ)$ act transitively on  $\bS^{2}(D)$). Now, by
\cite[Lemma 5.4.1]{EMV2010linnik} the latter holds for any $\mathbf{w}\in \mathbf{S}^{2}(D)$,
so $\phi$ is in fact a bijection from $M_{v}$ to $\mathbf{S}^{2}(D)$\footnote{Strictly speaking this is not needed but slightly
simplifies the argument in $\S$\ref{sub:Proof of mu_1,D} (cf.\ the
higher dimensionsal case in \cite{AES2014}).}.

To conclude the proof of~\eqref{ustake2} we show that if the first coordinate of $p(\cO_h)$ is $\mb{u}$ then the second one is $\br{\Lam_{\mb{u}}}$. Let $h\in M_{v}$ and denote $\gamma_{j}=\gamma_{j}(h),c_{j}=c_{j}(h)$
for $j=1,2$ so that
$
\cO_{h}=\Theta_{K}(k_{v}\gamma_{1},a_{v}k_{v}g_{v}\gamma_{2})\Gamma_{1}\times\Gamma_{2}.
$
Note that $e_{3}^{t}k_{v}\gamma_{1}=v^{t}\gamma_{1}=\pa{\gamma_{1}^{-1}v}^{t}$.
We denote $u=\gamma_{1}^{-1}v$. We need to show that
\begin{equation}
Ka_{v}k_{v}g_{v}\gamma_{2}\Gamma_{2}\stackrel{?}{=}[\Delta_{u}]=Ka_{u}k_{u}g_{u}\Gamma_{2}.\label{eq:Moved coset1}
\end{equation}
To see this note first that $a_{v}=a_{u}$ and that $k_{v}\gamma_{1}$ is a legitimate
choice of $k_{u}$. With these choices, (\ref{eq:Moved coset1}) (using the identity element of $K$ on both sides) will
follow once we show $g_{u}^{-1}\gamma_{1}^{-1}g_{v}\gamma_{2}\in\Gamma_{2}$.
The element $g_{u}^{-1}\gamma_{1}^{-1}g_{v}\gamma_{2}$ is certainly
a determinant 1 element which maps $\bR^{2}$ to itself. Furthermore,
the third entry of its third column is positive by the orientation
requirement in the definition of $g_{v}$ and $g_{u}$. Therefore,
it will be enough to show that this element maps $\bZ^{3}$ to itself.
Using that $\bZ=\widehat{\bZ}\cap\bQ\subset\bA_{f}$, we can see this
as follows:
\[
\begin{aligned}\bQ^{3}\supset g_{u}^{-1}\gamma_{1}^{-1}g_{v}\gamma_{2}\bZ^{3}=g_{u}^{-1}c_{1}^{-1}(c_{1}\gamma_{1}^{-1})g_{v}(\gamma_{2} & c_{2}^{-1})c_{2}\bZ^{3}=\\
=g_{u}^{-1}c_{1}^{-1}hg_{v}g_{v}^{-1}{h}^{-1}g_{v}c_{2}\bZ^{3} & =g_{u}^{-1}c_{1}^{-1}g_{v}c_{2}\bZ^{3}\subset\widehat{\bZ}^{3}.
\end{aligned}
\]

\eqref{ustake3} Recalling that $\mu_{\cO_{D}^{\infty}}=(\pi_{V_\bQ,\infty})_*(\mu_{\cO_D^\bA})$, we see that $\mu_{\cO_D^\infty}(\cO_h)$ is controlled by
$$\av{{\rm Stab}_{\bL_v(\bR\times \widehat{\bZ})}\bigl((e,h,e,g_{v}^{-1}hg_{v})\bG(\bQ)\bigr)}$$ which is independent of $h$ as $\bL_v$ is commutative. This together with~\eqref{ustake2} shows that $p_{*}(\mu_{\cO_{D}^{\infty}})$ is the normalized counting measure on its support. To show the same statement for $\nu_{D}$ we need to show that $\av{{\rm Stab}_{\Gamma_1}(Kk_v\gamma_1(h))}$ is independent of $h$. For large enough~$D$
this is clear since~$\Gamma_1$ is finite and every nontrivial~$\gamma\in\Gamma_1$ fixes
only two integer primitive points. The remaining cases can easily be checked (and are not
really important for us).
%This is shown in the proof of \cite[\S 5.2(3)]{EMLV11}.
\end{proof}

\subsection{From ${\rm ASL_{2}}$ to ${\rm SL_{2}}$ \label{sub:ASL2 to SL2}}

Let us \emph{momentarily }(see Remark \ref{rem:SL2 ASL remark}) denote
$\overline{\bG}_{2}={\rm SL_{2}}$ and let $\overline{X_{j}^{S}},\mu_{\overline{X_{j}^{S}}},\overline{\cO_{D}^{S}},\mu_{\overline{\cO_{D}^{S}}}$
be the analogous objects to the ones defined above. Note that $\overline{\bG}_{2}$
also has class number 1. Simplified version of the discussion above
implies analogous results for these analogous objects. In particular
we have:
\begin{cor}
\label{cor:S convergence implies SL2 theorem}In order to establish
the convergence (\ref{eq:SL2 conv mod Gamma1}) for a subset $A\subset\bN$,
it is enough to show that for some $\set{\infty}\subset S$, $\mu_{\overline{\cO_{D}^{S}}}$
equidistribute to $\mu_{\overline{X_{1}^{S}}}\otimes\mu_{\overline{X_{2}^{S}}}$
when $D\ra\infty,D\in A$.\end{cor}
\begin{rem}
\label{rem:SL2 ASL remark} Since in the rest of the paper we will only prove results
regarding $\SL_2$ and in order not to burden the notation we change the notation introduced above and denote the objects related to
$\SL_2$ without the over-line. For example,  from now on, $\bG_{2}={\rm SL_{2}}$.
\end{rem}

\section{Duke's Theorem and Joinings\label{sec:Duke and joinings}}

Choose any two distinct odd prime numbers $p,q$ and define $S_{0}=\set{\infty,p,q}$.
Let $\eta$ be a $\text{weak}^{*}$ limit of $\pa{\mu_{\cO_{D}^{S_{0}}}}_{D\in\bD(\set{p,q})\cap\bF}$
and let $\pi_{j}:X^{S_{0}}\ra X_{j}^{S_{0}}$ denote the
natural projections for $j=1,2$.
Corollary \ref{cor:S convergence implies SL2 theorem}
reduces the proof of Theorem~\ref{Thm: Main theorem} to the statement
that $\eta=\mu_{X_{1}^{S_{0}}}\otimes\mu_{X_{2}^{S_{0}}}$.
Roughly speaking, the latter will be obtained in two steps: the first, which
relies on Duke's Theorem, is to show that $\pa{\pi_{j}}_{*}\eta=\mu_{X_{j}^{S_{0}}}$
for~$j=1,2$.
The second uses \cite{EL2014} to bootstrap the information furnished
by the first step to deduce that $\eta=\mu_{X_{1}^{S_{0}}}\otimes\mu_{X_{2}^{S_{0}}}$
(and it is this final step that requires the splitting condition
at two places).
For both steps (but mainly for the second step) we will need the following preliminary lemma:

\begin{lem}
\label{lem:Invariant under tori}
Let~$\eta$ be a weak~$^*$ limit as above. There exist $0\neq v_{p}\in\bZ_{p}^{3},0\neq v_{q}\in\bZ_{q}^{3}$
and $g_{p}\in{\rm SL}_{3}(\bZ_{p}),g_{q}\in{\rm SL}_{3}(\bZ_{q})$
such that $\eta$ is invariant under a diagonalizable subgroup of
the form
\begin{equation}
\mathbf{T}\defi\set{\pa{h_{p},h_{q},g_{p}^{-1}h_{p}g_{p},g_{q}^{-1}h_{q}g_{q}}:\pa{h_{p},h_{q}}\in\bH_{v_{p}}(\bQ_{p})\times\bH_{v_{q}}(\bQ_{q})}.\label{eq:diagonally embedded tori}
\end{equation}
Furthermore, $\bH_{v_{\ell}}(\bQ_{\ell}),\,\ell=p,q$ are split tori, and
so $\bH_{v_{p}}(\bQ_{p})\times\bH_{v_{q}}(\bQ_{q})$ contains a group
isomorphic to $\bZ^{2}$ which is generated by an element~$a_p\in\bH_{v_p}(\bQ_p)$
with eigenvalues~$p,1,p^{-1}$ and an element~$a_q\in\bH_{v_q}(\bQ_q)$
with eigenvalues~$q,1,q^{-1}$.
\end{lem}

\begin{proof}
By Hensel's lemma any vector $v_{D}$
with $D\in\bD(S_{0})$ has the property that~$D=Q_0(v_D)\in -(\bZ_\ell^\times)^2$
for $\ell=p,q$. Moreover, $g_{v_{D}}\in{\rm SL_{3}}(\bZ_{\ell})$
for any prime $\ell$. We assume that $\eta$ is the $\text{weak}^{*}$ limit of
$\mu_{\cO_{D_{n}}^{S_{0}}}$ and let $v_{D_{n}}$ denote the integral
vector defining the orbit $\cO_{D_{n}}^{S_{0}}$.

For any prime $\ell$,
$\bZ_{\ell}^{3}$ and ${\rm SL_{3}}(\bZ_{\ell})$
are compact sets.
Thus we may choose a subsequence or, to simplify the notation, simply assume
that $\pa{v_{D_n}}$ converges in $\bZ_{p}^{3}$ to the vector~$v_p$,
in~$\bZ_q^3$
to~$v_{q}$,  $\pa{g_{v_{D_n}}}$
converges in ${\rm SL_{3}}(\bZ_p)$ to $g_{p}$, and in~${\rm SL}_3(\bZ_q)$ to $g_{q}$.
Note that $\cO_{D_n}^{S_{0}}$ admits a description, which is simliar to Proposition~\ref{prop: S packet}\eqref{ustake1}, as a union of $\Theta_K\times \bL_{v_{D_{n}}}(\bQ_{\set{p,q}})$-orbits. In particular $\mu_{\cO_{D_n}^{S_{0}}}$ is $\bL_{v_{D_{n}}}(\bQ_{\set{p,q}})$-invariant.
It readily follows that $\eta$ is invariant under the group appearing
in (\ref{eq:diagonally embedded tori}).

For the second assertion note that $\bH_{v_{\ell}}$ is the (split) orthogonal
group of the quadratic form $Q_{v_{\ell}}$
and that
 $Q_0(v_\ell)\in-(\bZ_\ell^\times)^2$ for~$\ell=p,q$.
Here~$Q_{v_\ell}$ is the isotropic (see the proof of~\eqref{eq: split condition}) quadratic form
on the orthogonal complement of~$v_\ell\in\bQ_p^3$.
The last assertion follows since $\bG_{1}(\bQ_{\ell})\cong{\rm PGL}_{2}(\bQ_{\ell})$ for $\ell=p,q$
and any split torus is conjugated to the diagonal group.
\end{proof}

\subsection{Two instances of Duke's Theorem.}\label{sub:Proof component duke }

In this section we prove the following proposition (which would hold
for any $S$ with $\infty\in S$):

\begin{prop}
\label{prop:Component wise duke} For $j=1,2$ let $\mu_{i,D}$ denote
the normalized probability measures on $\pi_{j}(\cO_{D}^{S_{0}})$.
Then $\mu_{i,D}$ equidistribute to $\mu_{X_{j}^{S_{0}}}$ when $D\ra\infty$ with~$D\in\bD(\set{p,q})\cap\bF$.
In particular, $\pa{\pi_{j}}_{*}\eta=\mu_{X_{j}^{S_{0}}}$ for~$j=1,2$.
\end{prop}

Both cases are special cases of the so-called Duke's Theorem \cite{Duke88}
and its refinements \cite{MH2006} (cf. \cite{MV2006ICM} where Theorem 1 there corresponds to $j=1$ and Theorem 2 to $j=2$).

\subsubsection{Proof for $j=1$.}\label{sub:Proof of mu_1,D}
As we wish to show equidistribution
on the $S_{0}$-adic space, we will use the formulation in \cite[\S 4.6]{EMLV11},
with $\textbf{G}=\bG_{1}={\rm SO_{3}}$ being the projectivized group of units
in the Hamiltonian quaternions.

Let $\mu$ be a $\text{weak}^{*}$ limit of a subsequence of $\mu_{1,D}$.
Lemma \ref{lem:Invariant under tori} implies that $\mu$ is invariant
under a product of two split tori $T=T_{p}\times T_{q}\subset\bG_{1}(\bQ_{p})\times\bG_{1}(\bQ_{q})$.
By \cite[\S 4.6]{EMLV11} $\mu$ is also invariant under $\bG_{1}(\bQ_{S_{0}})^{+}\defi\Pi\pa{\tilde{\bG_{1}}(\bQ_{S_{0}})}$
where $\Pi:\tilde{\bG}_{1}\ra\bG_{1}$ is the natural morphism from
the simply-connected cover of $\bG_{1}$. We will be done once we
show the following claim: $\bG_{1}(\bQ_{S_{0}})$ is generated by
$\bG_{1}(\bQ_{S_{0}})^{+}$ and $T$. To this end, note that $\tilde{\bG}_{1}(\bR)\ra\bG_{1}(\bR)$
is surjective and so by \cite[\S 8.2]{PR94} there exists a homomorphism
$\Psi$ mapping the group $\bG_{1}(\bQ_{S_{0}})/\bG_{1}(\bQ_{S_{0}})^{+}$
 to $S\defi\bQ_{p}^{\times}/\pa{\bQ_{p}^{\times}}^{2}\times\bQ_{q}^{\times}/\pa{\bQ_{q}^{\times}}^{2}$.
Furthermore, under the natural isomorphisms $\bG_{1}(\bQ_{\ell})\cong{\rm PGL_{2}}(\bQ_{\ell})$,
$\ell=p,q$, the coboundary map $\Psi$
is nothing but the determinant map. With this it is easy to verify that the torus $T$ is mapped surjectively
onto $S$. Hence the proposition follows for~$j=1$.

\subsubsection{Proof for $j=2$.\label{sub:Proof for mi_2,D}}

In this case, equidistribution follows from a subtler argument. For
more details on the classical number theory constructions we are considering
below see \cite[\S 5.2]{cohen1993course}. Recall that a binary quadratic
form $q=aX^{2}+bXY+cY^{2}$ over~$\bZ$ is called primitive if $(a,b,c)=1$ and
that ${\rm disc}(q)\defi b^{2}-4ac$. Primitivity and discriminant
are stable under the usual $\SL_{2}(\bZ)$-equivalence. Let ${\rm Bin}_{L}=\{[q]:{\rm disc}(q)=L\}$
denote the set of primitive positive definite binary quadratic forms
of discriminant $L<0$ considered up-to $\SL_{2}(\bZ)$-equivalence.
Finally recall that a number is called a fundamental discriminant
if it is the discriminant of the maximal order in a quadratic field.

\begin{claim}
\label{claim: primitive quadratic form}Let $v\in\bS^{2}(D)$. If
$D\equiv1,2\mod4$ then the two dimensional quadratic lattice $q_{v}\defi\pa{\Lambda_{v},x^{2}+y^{2}+z^{2}}$
defines an element in ${\rm Bin}_{-4D}$. If $D\equiv3\mod4$ then $q_{v}\defi\pa{\Lambda_{v},\frac12(x^{2}+y^{2}+z^{2})}$
defines an element in ${\rm Bin}_{-D}$. %In other words, in all cases~$q_v$ is an integral quadratic form with fundamental discriminant (and equals~$Q_v$ or~$\frac12Q_v$ depending on~$D$ mod~$4$).
\end{claim}

\begin{proof}[Proof of Claim \ref{claim: primitive quadratic form}]
The possible choices for an oriented basis of $\Lambda_{v}$
give rise to the ${\rm SL_{2}(\bZ)}$-equivalence of binary quadratic forms. For calculating the discriminant and show primitivity, we choose $v_{1},v_{2}$
as in the introduction as a $\bZ$-basis for $\Lambda_{v}$ and define
$Q_{v}$ to be the quadratic form $\pa{\Lambda_{v},x^{2}+y^{2}+z^{2}}$
with respect to this basis. That is, $Q_{v}=aX^{2}+bXY+cY^{2}$, where $a=(v_{1},v_{1}),\, b=2(v_{1},v_{2}),\, c=(v_{2},v_{2})$.
It follows from Equation (\ref{eq:covolume-1}) that $ac-\frac{b^{2}}{4}=D$
or ${\rm disc}(Q_{v})=-4D<0$.  By construction $Q_{v}$ is positive definite.

We will show that if $D\equiv3\mod4$ then $2|a$ and
$2|c$. Indeed, if $4\nmid b$ the equation $ac-\frac{b^{2}}{4}=D$
implies that $ac$ is divisible by $4$. The claim follows since $a$
and $c$ are sums of three squares so if $4|a$ or $4|c$ we will
have a contradiction to the primitivity of the vectors $v_{1}$ or
$v_{2}$. If $4|b$ then $ac\equiv3\mod4$. So without loss of generality
we may assume that $a\equiv3,c\equiv1\mod4$. This implies that all
the coordinates of $v_{1}$ are odd and exactly two of the coordinates
of $v_{2}$ are even. But then $\frac{b}{2}=(v_{1},v_{2})$ is odd,
which is a contradiction. Primitivity of $Q_{v}$ (resp. $\frac{1}{2}Q_{v}$
for $D\equiv3\mod4$) and the last statement of the claim
follow since for $D\in\bF$ we have ${\rm disc}(Q_{v}){\rm =disc}(\bQ(\sqrt{-D})$
(resp. ${\rm disc}(\frac{1}{2}Q_{v}){\rm =disc}(\bQ(\sqrt{-D})$
for $D\equiv3\mod4$), which implies the claim\footnote{The argument from \cite[Lemma 3.3]{AES2014}
could also be used to prove primitivity without the assumption $D\in\bF$.}.
\end{proof}

\begin{comment}
\begin{proof}
We first show that $\gcd(a,\frac{b}{2},c)=1$. Choose $\tilde{v}_{1}$
to be a primitive vector in $v^{\perp}\cap w^{\perp}\cap\bZ^{3}$
and choose $\tilde{v}_{2}\in v^{\perp}\cap\bZ^{3}$ such that $\tilde{v}_{1},\tilde{v}_{2},w$
are a (oriented) basis for $\bZ^{3}$. As $\tilde{v}_{1}\in w^{\perp}$
there exist a vector $u\in\bZ^{3}$ such that $(u,\tilde{v}_{1})=1$
and $u=A\tilde{v}_{1}+B\tilde{v}_{2}$. Let $\tilde{a}=(\tilde{v}_{1},\tilde{v}_{1}),\,\tilde{b}=2(\tilde{v}_{1},\tilde{v}_{2}),\,\tilde{c}=(\tilde{v}_{2},\tilde{v}_{2})$.
We have $A\tilde{a}+B\frac{\tilde{b}}{2}=1$ so $\gcd(\tilde{a},\frac{\tilde{b}}{2},\tilde{c})=1$
and therefore $\gcd(a,\frac{b}{2},c)=1$.

As $a$ and $c$ are sum of three squares, which are not all even,
$a$ and $c$ can be $1,2$ or $3$ modulo $4$. So $4|ac$ if and
only if $2|a$ and $2|c$. This implies that $q_{v}$ is primitive
if and only if $4\nmid ac$. If $D\equiv1\text{ or }2\mod4$ then
$ac-\frac{b^{2}}{4}=D$ implies that $4\nmid ac$. Assume now that
$D\equiv3\mod4$ and we need to show that $2|a$ and $2|c$.\end{proof}
\end{comment}

Due to Claim~\ref{claim: primitive quadratic form} we always set~$L=-4D$ if~$D\equiv1,2\mod4$
and~$L=-D$ if~$D\equiv3\mod4$.
Recall that $\cX_{2}\cong\Gamma_{2}\setminus\bH$ by sending $Kg\Gamma_{2}$
to $\Gamma_{2}g^{-1}.i\in\Gamma_{2}\setminus\bH$, where the action
on $i\in\bH$ is given by the regular M\"obius transformation. For $\alpha\in{\rm Bin}_{L}$
choose a quadratic form $q$ such that $\alpha=[q]$ and we denote
by $z_{q}$ the unique root of $q(X,1)$ belonging to the hyperbolic
plane $\bH$ and by $\mathbf{z}_{q}$ its $\Gamma_{2}$-orbit.
If~$q=\frac12Q$ (c.f.\ the case~$D\equiv 3\mod 4$ above), we may use the polynomial~$q(X,1)$
or~$Q(X,1)$ and obtain the same root --- we may also write~$\mathbf{z}_Q$ for the~$\Gamma_2$-orbit
of the root.
Finally, we define
$\mathbf{z}_{\alpha}=\mathbf{z}_{q}$ and note that this definition
does not depend on the choice of $q$ (within the~$\Gamma_2$-orbit). The set of Heegner points of
discriminant $L$ is $\cH_{L}\defi\set{\mathbf{z}_{\alpha}:\alpha\in{\rm Bin}_{L}}$.

\begin{claim}
\label{Claim:Heegner=00003Dlattice}Under the isomorphism $\cX_{2}\cong\Gamma_{2}\setminus\bH$
described above we have $\mathbf{z}_{q_{v}}=[\Lambda_{v}]$. \end{claim}
\begin{proof}[Proof of Claim \ref{Claim:Heegner=00003Dlattice}]
This follows from a straightforward calculation which is crucial
to the argument, so we carry it out in details. Recall that $\phi:{\rm ASL_{2}}\ra{\rm SL_{2}}$
denotes the natural projection and let $M_{v}=\phi(a_{v}k_{v}g_{v})$.
The claim will follow once we show that $M_{v}^{-1}.i=z_{Q_{v}}$
where $Q_{v}$ is the quadratic form w.r.t.$\ $the basis
$v_{1},v_{2}$ used to define $g_{v}$. To this end, let
$N_v=\smallmat{\alpha & \beta\\ \gamma & \delta}$
be the matrix whose columns are the first two entries of the vectors
$k_{v}v_{1},k_{v}v_{2}\in\bR^{3}$. As scalar matrices act trivially
as M\"obius transformations, the action of $a_{v}$ may be ignored and
also the cases $D\equiv3\mod4$ and $D\equiv1,2\mod4$ may be treated
uniformly. In other words, it is enough to show that $N_{v}^{-1}.i=z_{Q_{v}}$.
By the definition of $k_{v}$, the third entries of $k_{v}v_{1},k_{v}v_{2}\in\bR^{3}$
are zeroes, so we have the following equalities: $\alpha^{2}+\gamma^{2}=\norm{v_{1}}^{2}=a$,
$\beta^{2}+\delta^{2}=\norm{v_{2}}^{2}=c$ and $\alpha\beta+\gamma\delta=(v_{1},v_{2})=\frac{b}{2}$
and finally by (\ref{eq:covolume-1}) that $\det{N_v}=\alpha\delta-\beta\gamma=\sqrt{D}$.
The claim now follows since
\begin{multline*}
N^{-1}.i=\frac{\delta i-\beta}{-\gamma i+\alpha}=\frac{-\frac{b}{2}+i\sqrt{D}}{a}\\
=\frac{-b+\sqrt{-4D}}{2a}=\frac{-b+\sqrt{b^{2}-4ac}}{2a}=z_{Q_{v}}.
\end{multline*}
\begin{comment}
$N^{-1}.i=\frac{\delta i-\beta}{-\gamma i+\alpha}=\frac{\pa{\delta i-\beta}\pa{\gamma i+\alpha}}{-\gamma i+\alpha}=\frac{-\frac{b}{2}+i\sqrt{D}}{a}=\frac{-b+\sqrt{-4D}}{2a}=\frac{-b+\sqrt{b^{2}-4ac}}{2a}$
\end{comment}
\end{proof}

It is well-known \cite[5.2.8]{cohen1993course} that ${\rm Bin}_{L}$,
and therefore also $\cH_{L}$, is parametrized by $\cC_{D}\defi{\rm Pic}(R_{L})$,
the class group of the unique order $R_{L}\subset\bQ(\sqrt{-D})$ of discriminant $L$.
By Claim \ref{claim: primitive quadratic form},
in both cases (regarding the definition of~$L$ in terms of~$D$),
$\cC_{D}$ is the class group of $\bQ(\sqrt{-D})$.

%For $D\in\bF$ with $D\equiv1,2\mod4$ let ${\rm Bin}(D)={\rm Bin}_{-4D}$
%, $\cH(D)=\cH_{-4D}$ and $\cC_{D}={\rm Pic}(R_{-4D})$. Similarly,
%for $D\in\bF$ with $D\equiv3\mod4$ let ${\rm Bin}(D)={\rm Bin}_{-D}$,
%$\cH(D)=\cH_{-D}$ and $\cC_{D}={\rm Pic}(R_{-D})$.
Let
\begin{equation}
\cP_{D}\defi\set{\mathbf{z}_{q_{v}}:v\in\bS^{2}(D)}\subset\cH_L.\label{eq:def of Pd}
\end{equation}
Another instance of Duke's Theorem (see \cite[Theorem 2]{MV2006ICM})
implies that $\cH_L$ equidistribute on $\Gamma_{2}\setminus\bH$
when $D\ra\infty$, $D\in\bD\cap\bF$. If $\cP_L$ would always be equal
to $\cH_L$, we could conclude in the same way as we did in the case
$j=1$ above (e.g.\ using \cite[\S 4.6]{EMLV11}). However, this is not always
the case by the following claim.

\begin{claim}
\label{claim:squares}
Let $\cC_{D}^{2}$ be the subgroup of squares
in $\cC_{D}$. Under the above mentioned parametrization of $\cH_L$
in terms of the class group $\cC_{D}$
the set $\cP_{D}$ corresponds to a coset of $\cC_{D}^{2}<\cC_{D}$.
We further note that $\av{\cC_{D}^{2}}\asymp D^{\frac{1}{2}+o(1)}$ (and~$\cC_D^2=\cC_D$ if~$D$ is a prime).
\end{claim}

\begin{proof}
This is shown in \cite[\S 4.2]{EMV2010linnik} as we now explain.
Fix $D\in\bF\cap\bD$. As explained in \cite[\S 6]{EMV2010linnik},
and in fact is proven implicitly in Proposition \ref{prop: S packet},
the set $\mathbf{S}{}^{2}(D)$ is a
torsor\footnote{A \emph{torsor} of a group $G$ is a set on which $G$ acts freely
and transitively.%
} of $\cC_{D}$. Also, ${\rm Bin}_L$ is naturally a torsor of $\cC_{D}$.
Note that $\alpha_{\mathbf{v}}\defi [q_{v}]\in{\rm Bin}_L$ for $v\in\mathbf{v}$ is well-defined.
It is shown in \cite[\S 4.2]{EMV2010linnik} that under these torsors
structures, for any $\gamma\in\cC_{D},\mathbf{v}\in\mathbf{\bS}{}^{2}(D)$
we have
\[
\alpha_{\gamma.\mathbf{v}}=\gamma^{2}.\alpha_{\mathbf{v}}.
\]
Thus, it follows that the image of the map $\mathbf{v}\mapsto \alpha_{\mathbf{v}}$
in the torsor $\on{Bin}_L$
corresponds to a coset of $\cC_{D}^{2}$.
Thus, the same is true for $\cP_{D}=\set{\mathbf{z}_{q_{\mathbf{v}}}:\mathbf{v}\in\mathbf{S}^{2}(D)}$,
which is the corresponding image in $\cH_L$.

It is well-known \cite[(1.1)]{EMV2010linnik} that $\cC_{D},\mathbf{S}{}^{2}(D)$
and $\cH_L$ are asymptotically of size $D^{\frac{1}{2}+o(1)}$.
Gauss' genus theory \cite[Chapter 14.4]{cassels:RQF} tells us that
$[\cC_{D}:\cC_{D}^{2}]=2^{r(D)-1}$ where $r(D)$ is the number of
distinct primes dividing $D$. Thus we also have
$\av{\cC_{D}^{2}}=\av{\cP_{D}}\asymp D^{\frac{1}{2}+o(1)}$.
\end{proof}

We can now establish the desired equidistribution on~$X_2^{S_0}$.
Recall from Proposition \ref{prop: S packet} that~$p_*(\mu_{\cO_D^\infty})=\nu_D$.
Let~$\pi_2$ also denote the projection from~$X^{\infty}$ to~$X_2^\infty$,
and let~$\pi_K$ denote the projection from~$X_2^\infty$ to~$K\backslash X_2^\infty=\cX_2$.
By Claim \ref{Claim:Heegner=00003Dlattice} we further get that~$(\pi_K\circ \pi_2)_*\nu_D$
can be identified with the counting measure on $\cP_{D}\subset\cX_{2}\cong\Gamma_{2}\setminus\bH$.

Therefore, the equidistribution of $\pa{\pi_{K}\circ\pi_{S_{0},\infty}}_{*}\mu_{2,D}$
on $\cX_{2}$ is equivalent to the equidistribution of $\cP_{D}$
on $\cX_{2}$. The equidistribution of such\textbf{ }subsets, that
is, subsets corresponding to cosets of large enough subgroups was
established by \cite[Theorem 6]{MH2006} (see also \cite[Corollary 1.4]{HarcosThesis})
when $D\ra\infty$ along $\bD\cap\bF$. This equidistribution comes
in fact from a corresponding adelic statement. Since ${\rm SL_{2}}$
is simply-connected (and in particular has class number 1) the desired $S$-arithmetic
equidistribution for~$j=2$ follows from the proof of \cite[Theorem 6]{MH2006}.\hfill\qedsymbol

\begin{rem}
\label{Remark: not square free}
The only instance in which we use
the assumption that $D\in\bF$ is in the application of \cite[Theorem 6]{MH2006}.
Nevertheless it is known to experts that \cite[Theorem 6]{MH2006}
holds without the assumption that $D\in\bF$, but such statement does
not exist in print. A general adelic statement that will work for
all discriminants is planned to appear in an appendix by Philippe
Michel to an upcoming preprint (\cite{AKA3}) of the first named author.

We also remark that as we assume the congruence condition $D\in\bD(S_{0})$
both equidistribution statements , i.e.\ for $\mu_{1,D}$ and $\mu_{2,D}$,
may be deduced from the so-called Linnik's Method (as it is done in
slightly different context in \cite{ELMVDuke,EMV2010linnik}, see
in particular \cite[Prop.~3.6 (Basic lemma)]{ELMVDuke} which
only cares about the asymptotic size as in the last
statement of Claim~\ref{claim:squares}).
\end{rem}

\subsection{\label{sub:joinings}Joinings. }

From Proposition \ref{prop:Component wise duke} we know that $\pa{\pi_{j}}_{*}\eta=\mu_{X_{j}^{S_{0}}},\: j=1,2$
and that $\eta$ is a probability measure as $\pi_{2}$ has compact
fibers and $\eta(X^{S_{0}})=\mu_{X_{2}^{S_{0}}}(X_{2}^{S_{0}})=1$.
Furthermore, by Lemma \ref{lem:Invariant under tori}  $\eta$
is invariant under the group $\mathbf{T}$ that appears in (\ref{eq:diagonally embedded tori}).
This means that $\eta$
is a \emph{joining} for the action of $\mathbf{T}$ on the product
space $X_{1}^{S_{0}}\times X_{2}^{S_{0}}$. Our goal, which is to
show that $\eta=\mu_{X_{1}^{S_{0}}}\otimes\mu_{X_{2}^{S_{0}}}$, will
follow from \cite[Theorem 1.1]{EL2014}. Roughly speaking, it is shown
there that a joining for a higher rank action (this is the reason
we insist on $S_{0}$ to contain two primes) is always algebraic.
As $X_{1}^{S}$ is compact and $X_{2}^{S}$ is non-compact, the only
algebraic joining is given by the trivial joining. Below we will expand
this argument in greater detail, where we will be more careful regarding
the precise assumptions of \cite[Theorem 1.1]{EL2014}. To satisfy
these assumptions we need to reduce to the case where unipotents act
ergodically, where we have a diagonally embedded action of $\bZ^{2}$
by semisimple elements, and where the joining is ergodic. The precise
definitions will be given below.

We fix some ad-hoc notation for this proof. Let $G^{S_{0}}=G_{1}^{S_{0}}\times G_{2}^{S_{0}}$ and $\Gamma^{S_{0}}=\Gamma_{1}^{S_{0}}\times\Gamma_{2}^{S_{0}}$
where $G_{j}^{S_{0}}=\bG_{j}(\bQ_{S_{0}})$ and $\Gamma_{j}^{S_{0}}=\bG_{j}(\bZ^{S_{0}})$ for~$j=1,2$.
Finally let $G^{+}=G_{1}^{+}\times G_{2}^{S_{0}}$
where $G_{1}^{+}=\bG_{1}(\bQ_{S_{0}})^{+}$ is the (normal) open group
defined in $\S$\ref{sub:Proof of mu_1,D}. Using $G^{+}$ we decompose
$X^{S_{0}}$ into finitely many disjoint $G^{+}$-orbits
$\mathbf{X}_{r}\defi G^{+}(g_{r},e)\Gamma^{S_{0}},\, r\in R$ for
some $g_{r}\in G_{1}^{S_{0}}$ and an index set $R$.

By Proposition \ref{prop:Component wise duke} for $j=1$ we know
that
\begin{equation}
\eta(\mathbf{X}_{r})=\mu_{X_{1}^{S_{0}}}(G_{1}^{+}g_{r}\Gamma_{1}^{S_{0}})=\mu_{X^{S_{0}}}(G^{+}(g_{r},e)\Gamma^{S_{0}})>0.\label{eq:weight of each j}
\end{equation}
for all~$r\in R$. Now fix some $r\in R$ and define the probability measure $\eta_{r}\defi\frac{1}{\eta(\mathbf{X}_{r})}\eta|_{\mathbf{X}_{r}}$.
It follows that
\[
\pa{\pi_{1}}_{*}\eta_{r}=\mu_{1}^{r}\defi\frac{1}{\eta(\mathbf{X}_{r})}\left.\mu_{X_{1}^{S_{0}}}\right\vert_{G_{1}^{+}g_{r}\Gamma_{1}^{S_{0}}},
\]
where we may identify the latter with the normalized probability measure
$\mu_{G_{1}^{+}/\pa{G_{1}^{+}\cap g_{r}\Gamma_{1}^{S_{0}}g_{r}^{-1}}}$.
Also note that $\pa{\pi_{2}}_{*}\eta_{r}=\mu_{X_{2}^{S_{0}}}$, that
$G_{1}^{+}\cap\bG_{1}(\bR)=\bG_{1}(\bR)$ is connected,
and that $G_{1}^{+}\cap\bG_{1}(\bQ_{\{p,q\}})$
and $G_{2}^{S_{0}}$ are generated by one-parameter unipotent subgroups
(see e.g.$\ $\cite[\S 6.7]{BT73}). Furthermore, $G_{1}^{+}$ (resp.\ $G_{2}^{S_{0}})$
act ergodically on the quotient $G_{1}^{+}/\pa{G_{1}^{+}\cap g_{r}\Gamma_{1}^{S_{0}}g_{r}^{-1}}$
(resp.\ on $X_{2}^{S_{0}}$) with respect to their uniform measure.
This establishes one of the assumptions in \cite[Thm.~1.1]{EL2014} ---
in the terminology of~\cite{EL2014} the quotients $G_{1}^{+}/\pa{G_{1}^{+}\cap g_{r}\Gamma_{1}^{S_{0}}g_{r}^{-1}}$
and $G_{2}^{S_{0}}/\Gamma_{2}^{S_{0}}$ are ``saturated by unipotents''.

Let
\[
 A=\{(a_1(\mathbf{n}),a_2(\mathbf{n})):\mathbf{n}\in\bZ^2\}<\bG(\bQ_{\set{p,q}})
\]
be a subgroup isomorphic to~$\bZ^2$ as
in Lemma \ref{lem:Invariant under tori}. By construction~$a_2(\mathbf{n})= (g_{p}^{-1},g_{q}^{-1})a_{1}(\mathbf{n})(g_{p},g_{q})$ for all~$\mathbf{n}\in\bZ^2$.
Then, by Lemma \ref{lem:Invariant under tori} we have that $\eta$ is invariant
under $A$ and that $a(\mathbf{n})=(a_1(\mathbf{n}),a_2(\mathbf{n}))$
defines for~$\mathbf{n}\in\bZ^2$ a ``class-$\cA'$ homomorphism'', in the terminology
of \cite{EL2014}. Fix $r\in R$. As $G_{1}^{+}$ has
finite-index in $G_{1}^{S_{0}}$, it follows that there exists a finite-index
subgroup $\Lambda<\bZ^{2}$ (again isomorphic to~$\bZ^2$) such that $\eta_{r}$ is invariant
under $B=a(\Lambda)$. The restriction of $a$ to $\Lambda$ is also of class-$\cA'$.
This establish another assumption of \cite[Thm.~1.1]{EL2014}.

In general $\eta_{r}$ may not be~$B$-ergodic, but a.e.~ergodic component~$\eta_{r,\tau}$
(with~$\tau$ belonging to the probability space giving the ergodic decomposition)
will now satisfy all assumptions in \cite[Thm.~1.1]{EL2014}.
In fact $\eta_{r,\tau}$ is an ergodic ``joining for the higher rank action of $B=a(\Lambda)$''
and we may conclude that~$\eta_{r,\tau}$ is an algebraic joining. I.e.~$\eta_{r,\tau}$
is the Haar measure on a closed orbit of the form~$g_{r,\tau}M\Gamma$
where~$M$ is a finite index subgroup of a~$\bQ$-group~$\mathbb M<\bG_1\times\bG_2$
which projects onto~$\bG_j$ for~$j=1,2$.
However, as both~$\bG_1$ and~$\bG_2$ are simple~$\bQ$-groups whose
adjoint forms are different over $\bQ$ we obtain~$\mathbb M=\bG_1\times\bG_2$
and that $\eta_{r,\tau}=\mu_{1}^{r}\otimes\mu_{X_{2}^{S_{0}}}$ (for more details,
see the comment after \cite[Theorem 1.1]{EL2014}).
Using (\ref{eq:weight of each j}),
it now follows that $\eta=\mu_{X_{1}^{S_{0}}}\otimes\mu_{X_{2}^{S_{0}}}$
as we wanted to show.

\appendix

\section{the Associated Dirichlet Series \\By Ruixiang Zhang}\label{Appendix}

In this appendix, we look at a sum related to the study of the equidistribution in Theorem \ref{Thm: Main theorem}, and explain some facts about them from the scope of classical analytic number theory. Theorem \ref{Thm: Main theorem} has high dimensional twins, but we will concentrate on the theory in dimension $3$, since we already saw that this is the most interesting dimension. Parallel theories have been developed in the references for higher dimensions.

Let $\mathbb{H}$ be the usual upper half plane. Take $\phi$ on $\cX_{2} = \rm SL(2, \mathbb{Z}) \setminus \mathbb{H}$ to be a constituent of the spectrum decomposition, which can be a constant, a unitary Eisenstein series or a Maass cusp form, and then take a spherical harmonic $\omega$ on $\mathbb{R}^3$. Assume $k$ is the degree of the polynomial $\omega$. We form the following Weyl sum for each positive integer $n$:
\begin{equation}\label{Weylsum}
S(n, \omega, \phi) = \sum_{\mathbf{v}\in \mathbb{Z}^3_{\textrm{prim}}, \|\mathbf{v}\|^2 =n} \omega (\frac{\mathbf{v}}{\|\mathbf{v}\|}) \phi (z_{\mathbf{v}}).
\end{equation}

Here $z_{\mathbf{v}} \in \Gamma \setminus \mathbb{H}$ is defined as the following: Let the plane $b_{\mathbf{v}}$ be the orthogonal complement (with an orientation given by $\mathbf{v}$) of ${\mathbf{v}}$ and $L_{\mathbf{v}}$ the lattice consisting of all the integer points on $b_{\mathbf{v}}$. The shape of $L_{\mathbf{v}}$ corresponds to a point $z_{\mathbf{v}}\in \Gamma \setminus\mathbb{H}$ in the usual sense. In other words, $z_{\mathbf{v}}$ in this appendix will denote the Heegner point attached to $\mathbf{v}$ (previously defined by $z_{q_{\mathbf{v}}}$ in \S4.1.2).

The motivation of this sum is the joint equidistribution Conjecture \ref{Conj:ASL2 conjecture} in the paper. By a standard harmonic analysis argument (see the end), the pairs $(\frac{\mathbf{v}}{\|\mathbf{v}\|}, z_{\mathbf{v}})$ are jointly equidistributed if this sum, divided by the total number of $\mathbf{v}$'s, tend to zero (in some quantitative fashion) when either $\omega$ or $\phi$ is nontrivial.

As the first part of the appendix, we show that (\ref{Weylsum}) is familiar to number theorists. In fact, this sum $S(n, \omega, \phi)$ can be interpreted as the $n$-th coefficient of the Dirichlet series obtained by taking the special value at the identity of a maximal parabolic Eisenstein series on $\rm SL (3, \mathbb{Z}) \setminus \rm SL (3, \mathbb{R})$ formed with respect to $\phi$ and $\omega$. We now explain this correspondence and will state it as Theorem \ref{coeffofeisensteinthm}.

In $G = \rm SL (3, \mathbb{R})$, let the discrete subgroup $\Gamma = \rm SL (3, \mathbb{Z})$. Take a maximal parabolic subgroup $P \subseteq G$ to be

\begin{equation}
P = \left\{ \left( \begin{array}{ccc}
* & * & * \\
* & * & * \\
0 & 0 & *
\end{array} \right) \in G \right\}.
\end{equation}
According to the Langlands decomposition, we have $G = MANK$ where

\begin{displaymath}
M = \left\{ \left( \begin{array}{ccc}
* & * & 0 \\
* & * & 0 \\
0 & 0 & 1
\end{array} \right) \right\},
\end{displaymath}
\begin{displaymath}
A = \left\{ \left( \begin{array}{ccc}
a^{-\frac{1}{4}} & 0 & 0 \\
0 & a^{-\frac{1}{4}} & 0 \\
0 & 0 & a^{\frac{1}{2}}
\end{array} \right), a > 0 \right\},
\end{displaymath}
\begin{displaymath}
N = \left\{ \left( \begin{array}{ccc}
1 & 0 & * \\
0 & 1 & * \\
0 & 0 & 1
\end{array} \right) \right\}
\end{displaymath}
\begin{displaymath}
K = \rm SO(3, \mathbb{R}).
\end{displaymath}
Let
\begin{equation}
\Gamma_{\infty} = \left\{ \left( \begin{array}{ccc}
* & * & * \\
* & * & * \\
0 & 0 & 1
\end{array} \right) \in \Gamma \right\}.
\end{equation}
For an arbitrary $g \in G$ we can decompose it into $$g = m(g)a(g)n(g)k(g), m(g) \in M, a(g) \in A, n(g) \in N, k(g) \in K.$$ This decomposition may not be unique. However, it is easy to see that $a(g)$ is unique, the bottom row $\mathbf{v}(g)$ of $k(g)$ is unique. By abuse of notation we will use $\omega (k(g))$ to denote $\omega (\mathbf{v}(g))$, and use $a(g)$ to denote the bottom right entry of (the matrix) $a(g)$.  Moreover, $\phi (m (g))$ is well defined. It is also easy to verify that $\omega (k(g))$, $a(g)$ and $\phi (m (g))$ are invariant under the left multiplication by any element in $\Gamma_{\infty}$.

Therefore, for any $g \in G$, we form the sum
\begin{equation}
E(s, g, \omega, \phi) = \sum_{[\gamma] \in \Gamma_{\infty} \setminus \Gamma} \omega (k(\gamma g)) \phi (m (\gamma g)) a(\gamma g)^{-s}
\end{equation}
which is the maximal parabolic Eisenstein series we mentioned above.

Since all elements in $\Gamma$ have integral entries, when evaluated at the identity $g=I$, this series $E(s, I, \omega, \phi)$ become a Dirichlet series $\sum_{n=1}^{\infty} \frac{a_n}{n^s}$. We have

\begin{equation}
a_n = \sum_{[\gamma] \in \Gamma_{\infty} \setminus \Gamma : \text{ the third row of } \gamma \text{ has length } \sqrt{n}} \omega (k(\gamma g)) \phi (m (\gamma g)).
\end{equation}

We see some similarity between the summands of $a_n$ and $S(n, \omega, \phi)$. Actually we have the following

\begin{thm}\label{coeffofeisensteinthm}
\begin{equation}\label{relation}
a_n = S(n, \omega, \phi).
\end{equation}
\end{thm}

\begin{proof}
 First, note that all the primitive integer vectors of length $\sqrt{n}$ have a natural 1-1 correspondence to the last rows of $\gamma$ that have length $\sqrt{n}$, where $[\gamma] \in \Gamma_{\infty} \setminus \Gamma$. Thus with a primitive integer vector $\mathbf{v} = (a, b, c)$ of length $\sqrt{n}$ we associated 1 summand in both sides. It suffices to show that both summands associated with $\mathbf{v}$ are the same. It is obvious that the $\omega$ parts agree. We must show that the $\phi$ parts also agree. This is elementarily equivalent to the following statement: for the vector $\mathbf{v}$, the following two lattices have the same shape: (a) $\mathbb{Z}^3 \cap \mathbf{v}^{\perp}$ and (b) the projection of $\mathbb{Z}^3$ onto $\mathbf{v}^{\perp}$ (which are both easily seen to be lattices of rank $2$).

We now prove that for any vector $\mathbf{w} \in \mathbb{Z}^3 \cap \mathbf{v}^{\perp}$, there exists an integer vector $\mathbf{u}$ such that $\mathbf{w}=\mathbf{v}\times \mathbf{u}$. In fact, we can assume $\mathbf{w} = (f, g, h)$ and without loss of generality assume that $c \neq 0$. Then we must find integers $r, s, t$ such that $f = bt-cs, g= cr-at, h=as-br$. Note that by assumption we have $af + bg + ch = 0$. We deduce $\gcd(b, c) | af$. Since $(a, b, c)$ is primitive, we have $\gcd(a, \gcd(b, c)) = 1$ and thus $\gcd(b, c) | f$. Hence we can choose $t$ such that $c| f-bt$. In this situation $c| af+bg -a(f-bt)$, or $c| b(g+at)$. Hence we can change $t$ by a multiple of $\frac{c}{\gcd (b, c)}$, if necessary, to make both $c| f-bt$ and $c| g+at$. Now we just set $s = \frac{bt-f}{c}$, $r = \frac{g+at}{c}$. It is then easy to deduce $h = as - br$ by the fact that $af + bg + ch = 0$.

By the last paragraph, the entire lattice (a) is the cross product of the lattice (b) and the vector $\mathbf{v}$. Hence they have the same shape and the theorem is proved.
\end{proof}

It is not surprising that one could use the Dirichlet series $E(s, I, \omega, \phi)$ to study the analytic properties of the Weyl sum $a_n$, which would be naturally required if one wants to remove the congruence conditions of Theorem \ref{Thm: Main theorem} and prove Conjecture \ref{Conj:ASL2 conjecture}.

To address the problem, we need a good estimate for all individual coefficients $a_n$. The work of Gauss (see e.g. \cite{venkov1970elementary} for a nice account) shows that the total number of summands in $a_n$ is given by the following theorem.

\begin{thm}[Gauss\cite{gauss1889untersuchungen}]
Given an integer $n>3$. The number of coprime integer solutions $(x, y, z)$ to the equation $n= x^2 + y^2 + z^2$ is $12h$ for $n \equiv 1, 2 \mod 4$, and is $24h'$ for $n \equiv 3 \mod 8$, where $h$ and $h'$ are the number of properly and improperly primitive classes of positive forms of determinant $-n$.
\end{thm}

By Siegel's theorem together with Dirichlet's class number formula (see \cite{davenport1967multiplicative}), $h$ and $h' \gg_{\eps} n^{\frac{1}{2}-\eps}$, and are usually around $n^{\frac{1}{2}}$ (with an arbitrary small error on the exponent). So we would like to have a power saving from the exponent $\frac{1}{2}$ for all $a_n$.

\begin{rem}
We note that it suffices to consider \emph{even} forms $\omega$ and $\phi$. Otherwise it is obvious that $a_n = 0$. We assume this is the case for the rest of the discussion.
\end{rem}

It is still not clear how to do this in the greatest generality. But there have been partial results. Part of the following brief account already appeared in \S4.1 but we recall it once more for the reader's convenience. In the special case $\phi = 1$, this reduces to a situation that could be treated with Iwaniec's celebrated estimation of Fourier coefficients for half-integral weight holomorphic modular forms \cite{Iwaniec87}, as the series become a theta series. For the case $\omega = 1$ again by the work of Gauss \cite{venkov1970elementary} we know we are summing the $\phi$ over a certain genus of quadratic forms of determinant $-n$. If we pretend that we are summing over all the quadratic forms (CM points), this can be settled using Duke's generalization of Iwaniec's argument to non-holomorphic forms \cite{Duke88}. Using Waldspurger's formula and subconvexity the power saving for the real problem is also known \cite{MH2006}.

For general $\phi$ and $\omega$, the connection of the series $E(s, I, \omega, \phi)$ to modular forms is still mysterious and we currently do not have the desired power saving. Nevertheless, since it is a specific value (meaning for the fixed $g = I$) of an Eisenstein series, the analytic continuation and the functional equation are known (see e.g. \cite{terras1982automorphic}). Interestingly, Maass, when originally dealing with this very equidistribution problem (but in the ball, not on the sphere), also deduced these analytic properties \cite{Maass1959} \cite{Maass1971}. The following theorem could be easily derived from Maass's work in \cite{Maass1971} (see Chapter 16).

\begin{thm}\label{Maassthm}
Let $\Xi (s, \omega, \phi) = E (s, I, \omega, \phi) \Lambda^* (s, \phi)$. Then $\Xi (s, \omega, \phi)$ is holomorphic on $\mathbb{C}$ except for a possible pole at $s = \frac{3}{2}$. The pole exists if and only if both $\omega$ and $\phi$ are trivial. Also $\Xi (s, \omega, \phi)$ satisfies the functional equation
\begin{equation}
\Xi (\frac{3}{2} - s, \omega, \phi) = \Xi (s, \omega, \phi).
\end{equation}

Here $\Lambda^* (s, \phi)$ is a kind of ``completed $L$-function of $\phi$'' which we now define. Assume $\lambda (1- \lambda)$ is the eigenvalue of the Laplacian $\Delta = - y^2 (\frac{\partial^2}{\partial x^2} + \frac{\partial^2}{\partial y^2})$ for $\phi$. When $\phi$ is a constant or unitary Eisenstein series, $\Lambda^* (s, \phi)$ factorizes:
\begin{equation}
\Lambda^* (s, \phi) = \Lambda (2s-\lambda, \zeta) \Lambda (2s-(1-\lambda), \zeta)
\end{equation}
where $\Lambda (s, \zeta) = \pi^{-\frac{s}{2}} \Gamma (\frac{s}{2}) \zeta (s)$ is (up to normalization) the completed Riemann zeta function. When $\phi$ is a Maass cusp form, $\Lambda^* (s, \phi) = \Lambda (2s - \frac{1}{2}, \phi)$. Where $\Lambda (s, \phi)$ is the usual completed $L-$function of $\phi$ (we use the fact that $\phi$ is even):
\begin{equation}
\Lambda (s, \phi) = \pi^{-s} \Gamma (\frac{s}{2} + \frac{\lambda}{2} - \frac{1}{4})  \Gamma (\frac{s}{2} - \frac{\lambda}{2} + \frac{1}{4}) L(s, \phi).
\end{equation}
\end{thm}

With only the analytic properties stated in Theorem \ref{Maassthm}, we cannot expect to get a good control of the size of each individual term $|a_n|$. However we can get an ``average bound'' for the sum of $a_n$. Next we will show one such bound.

We prove the following:

\begin{thm}
\begin{equation}\label{averagethm}
\sum_{n \leq X} a_n = \left( \frac{1}{Area(\bS^2)}\int_{\bS^2} \omega \right)\left(\frac{1}{Area(\cX_{2})}\int_{\cX_{2}} \phi \right)X^{\frac{3}{2}} + O_{\omega, \phi, \eps}(X^{\frac{15}{14}+ \eps}).
\end{equation}

Obviously, the product of the integrals (main term) in (\ref{averagethm}) is nonzero if and only if both $\omega$ and $\phi$ are trivial.
\end{thm}

\begin{proof}
We fix $\omega, \phi$ from the beginning. By lattice points counting, the theorem is obvious when $\omega$ and $\phi$ are trivial. Next we assume that this is not the case. To be explicit, let's assume $\phi$ is a Maass cusp form. Other cases are similar. We invoke a theorem, which is a general bound about sums of coefficients of Dirichlet series. We state the theorem in the form we need.

\begin{thm}[A special case of Theorem 4.1 in \cite{chandrasekharan1962functional}]
Assuming we have two Dirichlet series $f (s) = \sum_{n=1}^{\infty} \frac{c_n}{n^s}$ and $g (s) = \sum_{n=1}^{\infty} \frac{d_n}{n^s}$ and a product of Gamma factors $\Delta (s) = \prod_{\nu = 1}^N \Gamma (\alpha_{\nu} s + \beta_{\nu})$, satisfying the functional equation
\begin{equation}
\Delta(s) f (s) = \Delta (M-s) g(M-s)
\end{equation}
for some $M>0$. Also assume $f$ is entire. Then
\begin{equation}\label{KNestimate}
\sum_{n \leq X} c_n = O(X^{\frac{M}{2}- \frac{1}{4A}+2A\eta u}) + O(\sum_{X<n\leq X'} |c_n|).
\end{equation}
In (\ref{KNestimate}), $A = \sum_{\nu = 1}^N \alpha_{\nu} \geq 1$, $X' = X+O(X^{1-\eta-\frac{1}{2A}})$, $u = \beta - \frac{M}{2} -\frac{1}{4A}$ where $\beta$ satisfies $\sum_{n=1}^{\infty} \frac{|d_n|}{n^{\beta}} < \infty$, $\eta$ is any positive number at our disposal.
\end{thm}

Take $f = g = E(s, I, \omega, \phi)L(2s-\frac{1}{2}, \phi)$. Maass proved $E(s, I, \omega, \phi)L(2s-\frac{1}{2}, \phi)$ is entire \cite{Maass1971}, which enables us to do the substitution. We see here $M = \frac{3}{2}$. By Theorem \ref{Maassthm}, $E(s, I, \omega, \phi)L(2s-\frac{1}{2}, \phi)$ has a functional equation with two Gamma factors, meaning $A=2$.  Finally we determine $\beta$. Assume $L(s, \phi) = \sum_{n=1}^{\infty} \frac{b_n}{n^s}$. Then $L(2s-\frac{1}{2}, \phi) = \sum_{n=1}^{\infty} \frac{b_n \sqrt{n}}{n^{2s}}$. Hence $|c_n| = |d_n| \leq \sum_{m^2 k = n} |a_k| |b_m| \sqrt{m} \leq \sum_{m^2 k = n} k^{\frac{1}{2} + \eps} m \ll n^{\frac{1}{2} + \eps}$. We can take $\beta = \frac{3}{2} + \eps$ and thus $u = \frac{5}{8} + \eps$.

We conclude that $\sum_{n \leq X} c_n = O(X^{\frac{5}{8}+ (\frac{5}{2} + \eps) \eta}) + O(X^{\frac{5}{4} + \eps - \eta})$. After an optimization we get $\sum_{n \leq X} c_n = O(X^{\frac{15}{14} + \eps})$. Now if $\frac{1}{L(s, \phi)} = \sum_{n=1}^{\infty} \frac{h_n}{n^s}$, from the Euler product of $L(s, \phi)$ we easily get $|h_n| \ll n^{\frac{1}{2}}$. Hence

\begin{equation}
|\sum_{n \leq X} a_n| = |\sum_{n \leq X} h_n \sqrt{n} \sum_{mn^2 \leq X} c_m| \ll |\sum_{n \leq X} |h_n| \sqrt{n} (\frac{X}{n^2})^{\frac{15}{14}+\eps}|\ll X^{\frac{15}{14} + \eps}.
\end{equation}
\end{proof}

We end the discussion with some further remarks concerning this approach. All the dependencies on $\omega$ and $\phi$ are polynomial in terms of their eigenvalues and can be made explicit by a slightly more careful treatment. It is then standard to do a spectral decomposition (of the smoothed characteristic function of the underlying domain) and take the Weyl law (see \cite{selberg1991collected}) into account, to get an estimate of the remainder term needed for the joint equidistribution result ``in a big ball'' --- pairs $(\frac{\mathbf{v}}{\|\mathbf{v}\|}, z_{\mathbf{v}})$ get jointly equidistributed for $\|\mathbf{v}\| \leq n$ when $n \rightarrow \infty$. We can also get the joint equidistribution ``in a thinner shell'' (for some $X^{1-\theta} < \|\mathbf{v}\| < X$) where $\theta > 0$ depends on the remainder term we have. But the conjectured joint equidistribution result ``on every sphere of a reasonable radius'' (Conjecture \ref{Conj:ASL2 conjecture}) requires new ideas.

This type of (quantitative) bounds for the remainder term of the averaged size of $a_n$ were also obtained by elementary methods by Schmidt \cite{WSchmidt-sublattices}, in more general cases. Our approach will also have corresponding generalization for higher dimensional settings.

%\bibliographystyle{plain}
%\bibliography{SOS}

\end{document}